\documentclass[12pt]{article}
\usepackage[margin=2.5cm]{geometry}
\usepackage{amsmath,amsthm,amsfonts,amssymb}
\usepackage[mathscr]{euscript}
\usepackage{amsmath}
\usepackage[table]{xcolor}
\usepackage{graphicx,latexsym}
\usepackage{lscape}
\usepackage{fixmath}
\usepackage{multicol}
\usepackage{graphicx}
\usepackage{caption}
\usepackage{float}
\usepackage{cite}
\usepackage{subfig}
\usepackage{setspace}
\usepackage{xcolor}
\usepackage{multirow}
\usepackage{hhline} 
\usepackage[utf8]{inputenc}
\usepackage{gensymb}
\usepackage{setspace}
\usepackage{abstract}
\usepackage{caption}

\newtheorem{theorem}{Theorem}[section]

\newtheorem{observation}[theorem]{Observation}

\newtheorem{property}[theorem]{Property}

\newtheorem{remark}[theorem]{Remark}

\newcommand{\red}{\color{red}}

\newcommand{\cp}{\,\square\,}

\newcommand{\diam}{{\rm diam}}

\newcommand{\AC}{{\rm AC}}
\newcommand{\OC}{{\rm OC}}
\newcommand{\AP}{{\rm AP}}
\newcommand{\OP}{{\rm OP}}

\newcommand{\VT}{{\rm VT}}

\def\cp{\,\square\,}

\begin{document}
	
\title{Properties of Villarceau Torus}
	
\author{
	Paul Manuel
}

\date{}

\maketitle
\vspace{-0.8 cm}
\begin{center}
	Department of Information Science, College of Life Sciences, Kuwait University, Kuwait \\
	{\tt pauldmanuel@gmail.com, p.manuel@ku.edu.kw}
		
\end{center}

\begin{abstract}
{\red This is an initial draft of ongoing research. This incomplete research is submitted to Kuwait University for research funding. The project proposal is under processing. Please do not hesitate to write your comments to improve the quality of the paper. You will reach me at the email address stated above. I will be pleased to interact with you.}
\end{abstract}


\medskip
\noindent{\bf AMS Subj.\ Class.: 05C12, 05C70, 68Q17}

\section{Introduction}
One interesting geometrical problem related to celestial surfaces posed by 19th century mathematicians was "Given an arbitrary point P on a surface, how many distinct circles can be drawn on the surface passing through the point P?". This was called the n-circle property. 
While a sphere has the infinite circle property, it was shown that some surfaces have the n-circle property for n equal to 4, 5 or 6 \cite{Blum80}. A startling conjecture by Richard Blum \cite{Blum80} is that there are no surfaces with the n-circle property for n greater than 6 and less than infinity. 
A surface which was attracted by 19th century mathematicians was torus. Until 1848, torus was known only by two circles which are toroidal and poloidal circles. In 1848, Yvon Villarceau demonstrated two diagonal circles on a torus \cite{Roegel14, Vill48}. 
The toroidal circle of a torus is horizontal and lies in the plane of the torus. The poloidal circle is vertical and is perpendicular to the toroidal. The third and fourth circles are obliquely inclined with respect to the first two circles of the torus and are known as the Villarceau circles \cite{Roegel14}. Refer to Figure \ref{fig:Torus_4Circles}. In other words, Villarceau circles are a pair of circles produced by cutting a torus obliquely through the center at a special angle $\theta$ \cite{Roegel14}. At the given point P on the torus, while the third circle makes an acute angle with the horizontal plane, the fourth one makes an obtuse angle with horizontal plane \cite{Roegel14}. 
Following Villarceau circles, there were two types of tori.  A torus built by toroidal and poloidal circles is called grid (ring) torus and a torus built by the Villarceau circles is called Villarceau torus. Refer to Figure \ref{fig:RingTorus_VT}.  In Physics, it is called spiral torus \cite{FeCh16, GuCh02}. The $2$-dimensional representation of Villarceau torus is shown in Figure \ref{fig:BW_Two_Models_4x5}.   
\begin{figure}[ht!]
	\begin{center}
		\scalebox{1.0}{\includegraphics{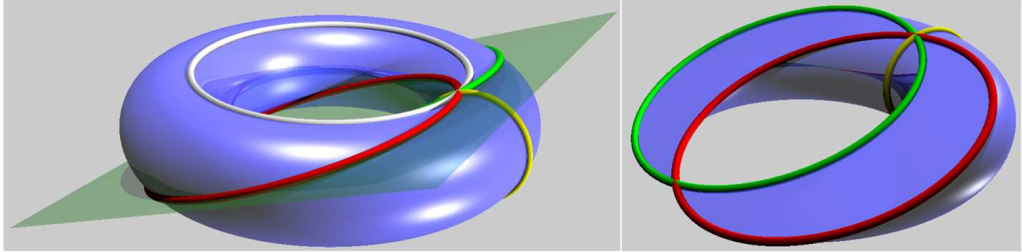}}
	\end{center}
	\caption{The toroidal circle (white color) of a torus is horizontal and lies in the plane of the torus. The poloidal circle (yellow color) is vertical and is perpendicular to the toroidal. The third and fourth circles (red and green colors) are obliquely inclined with respect to the first two circles of the torus and are known as the Villarceau circles. The diagrams are retrieved from \cite{Roegel14}.}
	\label{fig:Torus_4Circles}
\end{figure}
\begin{figure}[ht!]
	\begin{center}
		\scalebox{1.5}{\includegraphics{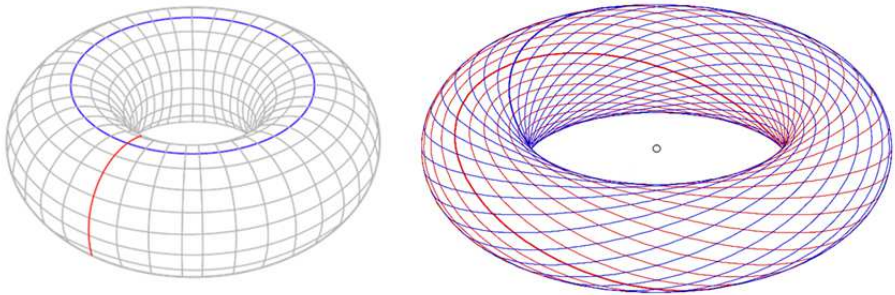}}
	\end{center}
	\caption{Ring torus and Villarceau torus. The diagrams are extracted from Wikimedia Commons.}
	\label{fig:RingTorus_VT}
\end{figure}
\begin{figure}[ht!]
	\begin{center}
		\scalebox{0.9}{\includegraphics{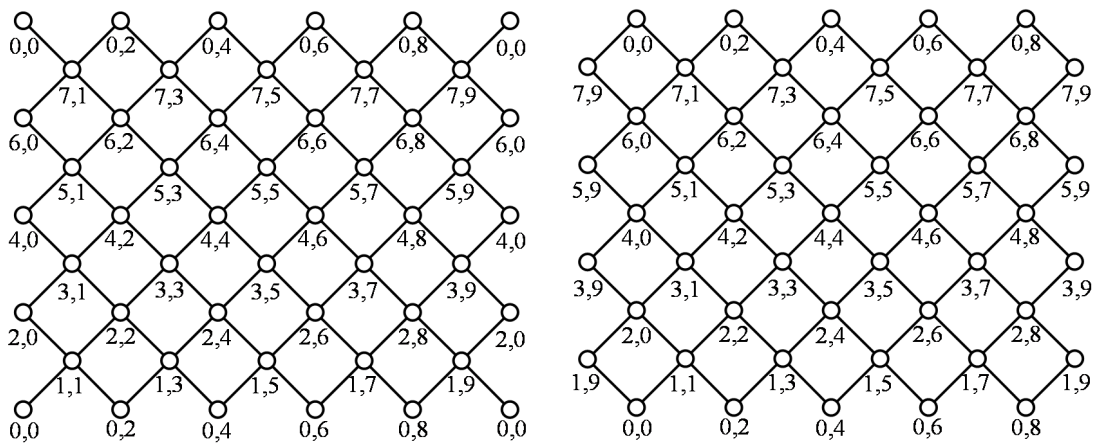}}
	\end{center}
	\caption{The 2D representation of Villarceau torus $\VT(4,5)$. One is a cyclic rotation of the other.}
	\label{fig:BW_Two_Models_4x5}
\end{figure}
\section{Terminologies and some observations of Villarceau torus $\VT(r,s)$}
\label{sec:properties-VT}
In Discrete Mathematics, $(Z_n, \oplus_n)$ represents the additive group of integers modulo $n$ where $Z_n = \{0,1 \ldots n-1\}$ and $a\oplus_n b = (a+b)\mod n$. 
The $2$-dimensional representation of Villarceau torus which are the projection of $3$-dimensional torus onto a $2$-dimensional plane was studied by several authors \cite{Baird18, Dorst19, MoMo11}. Refer to Figure \ref{fig:BW_Two_Models_4x5}.  The Villarceau torus is denoted by $\VT(rs)$ where the addition  in both coordinates is  modular addition. The addition of two vertices $(r_1, s_1)$ and $(r_2, s_2)$ in $\VT(r,s)$ is defined as follows:
$$(r_1, s_1) + (r_2, s_2) = (r_1+ r_2, s_1+ s_2) = (r_1 \oplus_{2r} r_2, s_1 \oplus_{2s} s_2)$$

The vertex set $V$ of $\VT(r,s)$ is 
$V$ = $\{(2i+1, 2j+1) : 0 \leq i \leq r-1, 0 \leq j \leq s-1\} \cup \{(2i, 2j) : 0 \leq i \leq r-1, 0 \leq j \leq s-1\}$.
The edge set $E$ of $\VT(r,s)$ is 
$E = \{(2i, 2j)(2i+1, 2j+1), (2i, 2j)(2i+1, 2j-1) : 0 \leq i \leq r-1, 0 \leq j \leq s-1\}$. See Figure \ref{fig:BW_Two_Models_4x5}.

A subgraph $H$ in $G$ is \textit{isometric}, if $d_H(u,v)=d_G(u,v)$ for every pair $u$ and $v$ of vertices of $H$. A set S of vertices in $G(V,E)$ is \textit{convex} if for each pair $x$ and $y$ of vertices in $S$, each isometric $x,y$-path entirely  lies inside $S$.
A cycle or path is \textit{isometric} (resp. \textit{convex}) if its induced subgraph is isometric (resp. convex). 
Here are some facts on convex paths and cycles which we recall to refresh our knowledge. A convex cycle (resp. path) is an isometric cycle (resp. path). Even cycles are not convex but odd cycles are convex. In an even cycle, a diametral path is not convex. However, in an odd cycle, every path is convex. Now we list some terminologies and notations which we use frequently in this paper.

While a path and a cycle are denoted by $P$ and $C$ respectively, an isometric path and an isometric cycle are denoted by $IP$ and $IC$ respectively. 
The edges in $VT(r,s)$ are either acute or obtuse \cite{Coxeter69, Dorst19, MoMo11}. See Figure \ref{fig:BW_Two_Models_4x5}. The set of acute edges of $\VT(r,s)$ is denoted by $E_a$ and the set of obtuse edges of $\VT(r,s)$ is denoted by $E_o$. In other words, the edge set $E$ of $\VT(r,s)$ is partitioned by acute edges and obtuse edges such that $E = E_a \cup E_o$, $E_a \cap E_o = \emptyset$, and $|E_a|=|E_o|$.
A path in $\VT(r,s)$ is called an \textit{acute} (resp. \textit{obtuse}) \textit{path}  if it contains only acute (resp. obtuse) edges. 
In the same way, a cycle in $\VT(r,s)$ is called an \textit{acute} (resp. \textit{obtuse}) \textit{cycle} if it contains only acute (resp. obtuse) edges.   
An acute (resp. obtuse) path  is represented by $\AP$ (resp. $\OP$). In the same way, an acute (resp. obtuse) cycle is represented by $\AC$ (resp. $\OC$).

\noindent 
Now, we define acute and obtuse paths of $\VT(r,s)$ as follows:

$\AP(0,2k)= (0,2k)(1,2k +1) \ldots (i,2k + i) \ldots (2r-1, 2k + (2r-1))(0, 2k + 2r)$, for $k = 0,1 \ldots s-1$.

$\OP(0,2k) = (0,2k)(1,2k-1) \ldots (i, 2k-i) \ldots (2r-1, 2k-(2r-1))(0, 2k-2r)$, for $k = 0,1 \ldots s-1$.

\noindent
The path $\AP(0,2k)$ consists of only acute edges and it is an acute path. Similarly, $\OP(0,2k)$ is an obtuse path. 
Now, we define acute cycle $\AC(0,2k)$ and obtuse cycle $\OC(0,2k)$ of $\VT(r,s)$. Cycle $\AC(0,2k)$ extends $\AP(0,2k)$ along the acute edges until it reaches $(0,2k)$ as follows:

$\AC(0,2k) = (0,2k)(1,2k +1) \ldots (i,2k + i) \ldots (2r-1, 2k + (2r-1))(0, 2k+2r) \ldots (2r-1,2k-1)(0,2k)$, for $k = 0,1 \ldots s-1$.

The above definition is well-defined because vertex $(0,2k)$ has only two acute edges $(0,2k)(1,2k +1)$ and $(2r-1, 2k-1)(0, 2k)$. If a cycle exits at $(0,2k)$ through $(0,2k)(1,2k+1)$, then it will reach back to $(0,2k)$ only through $(2r-1, 2k -1)(0, 2k)$. Moreover, The cycle $\AC(0,2k)$ consists of only acute edges and thus it is an acute cycle.

In the same way, $\OC(0,2k)$ extends the obtuse path $\OP(0,2k)$ along the obtuse edges until it reaches $(0,2k)$:

$\OC(0,2k) = (0,2k)(1,2k-1) \ldots (i, 2k-i) \ldots (2r-1, 2k-(2r-1))(0, 2k-2r) \ldots (2r-1,2k+1)(0,2k)$, for $k = 0,1 \ldots s-1$.

A graph $G$ is said to be \textit{uniform geodesic} if all the maximal isometric paths of $G$ are of uniform length $\diam(G)$. The striking difference between ring torus $C_r\cp C_s$ and Villarceau torus $\VT(r,s)$ are the results given below:

\begin{property}
	\label{pro:VT_diameter} Let  $\VT(r,s)$ denote a Villarceau torus.
		\begin{enumerate}
		\item	While the diameter of $C_r\cp C_s$ is $r+s$, the diameter of $\VT(r,s)$ is $\max\{r,s\}$.
		\item  While ring torus is uniform geodesic, Villarceau torus is not uniform geodesic. Villarceau torus $\VT(r,s)$ has maximal isometric paths of length $r$ as well as  maximal isometric paths of length $s$.
	\end{enumerate}
\end{property}

We shall use the following facts at a later stage:
\begin{property}
	\label{pro:VT_trivial_properties}
	Let  $\VT(r,s)$ denote a Villarceau torus where $s$ is assumed to be larger than $r$.
	\begin{enumerate}
		\item The vertex set $V$ and the acute edge set $E_a$ of $VT(r,s)$ are partitioned by paths $\{\AP(0,2k) : k = 0,\ldots, s-1\}$. In the same way, the vertex set $V$ and the obtuse edge set $E_o$ of $\VT(r,s)$ is partitioned by paths $\{\OP(0,2k) : k = 0,\ldots, s-1\}$. 
		\item The cycles $\AC(x,y)$ $($resp. $\OC(x,y))$ are unique because they are restricted to running through only acute (resp. obtuse) edges. In other words, given a vertex $(x,y)$ in $\VT(r, s)$, there exists only one acute cycle $\AC(x,y)$ and only one obtuse cycle $\OC(x,y)$ passing through $(x,y)$. 
	\end{enumerate}
\end{property}
\section{Toroidal helices in $\VT(r,s)$}
\label{sec:cycles-VT}
A toroidal helix is a coil wrapped around a torus. It revolves around the poloidal direction as well as toroidal direction and is characterized by the number of revolutions it winds around the torus \cite{YoFi15, ZhNa20}. Moreover, the wrapping angle of a toroidal helix against all the toroidal and poloidal cycles of the torus is a constant angle $\theta$ \cite{OlBo12, YoFi15}. In this section, we study the mathematical properties of the toroidal helical motion in the magnetic field by means of Villarceau torus. Refer to Figure \ref{fig:Pysics_Toroidal_Helix}. behind Since a toroidal helix flows around poloidal and toroidal direction, there are two types of revolutions \cite{YoFi15}: 
\begin{enumerate}
	\item \textit{Poloidal Revolution} which spirals around the poloidal direction.
	\item \textit{Toroidal Revolution} which spirals around the toroidal direction. 
\end{enumerate}
\begin{figure}[ht!]
	\begin{center}
		\scalebox{0.9}{\includegraphics{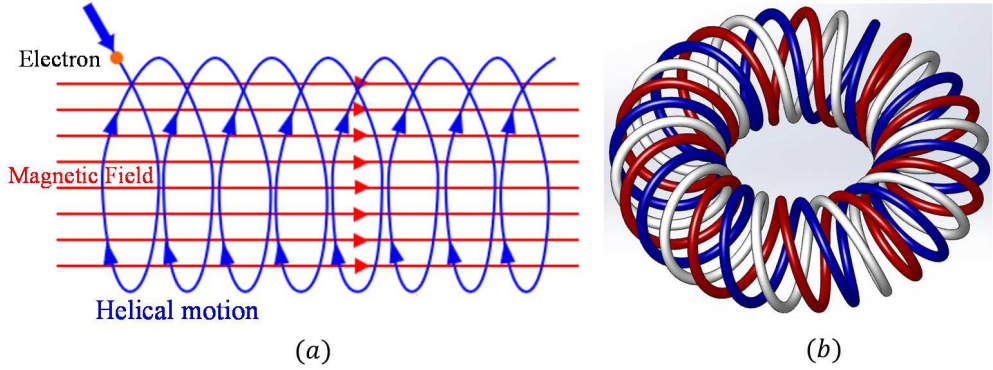}}
	\end{center}
	\caption{$(a)$ The diagram is extracted from  
		https://www.schoolphysics.co.uk
		$(b)$ Toroidal helices in $\VT(r,s)$. The diagram is extracted from https://grabcad.com/library/toroidal-helix-1.}
	\label{fig:Pysics_Toroidal_Helix}
\end{figure}
Here is a basic result in number theory which is available in most textbooks on number theory. 
\begin{theorem} {\rm \cite{HeCl02}}
	\label{thm:gcd}
	Given two integers $a$ and $b$ such that $b$ is larger 
	than $a$, let $d=\gcd(a,b)$, $p = a/d$ and $q = b/d$. 
	Then $V = \{0,1,\ldots,b-1\}$ is partitioned into 
	$V_0, V_1 \cdots, V_{d-1}$ where $V_0 = \{0, a \cdots (q-1)\cdot a\}$, 	$V_1 = \{1, 1+a \cdots 1+(q-1)\cdot a\}$ $\cdots$, $V_{d-1} = \{(d-1), (d-1)+a, \cdots (d-1)+(q-1)\cdot a\}$.
\end{theorem}
The above result is applied in the following two theorems.
\begin{theorem} {\rm \cite{Meijer91}}
	\label{thm:Cir_Graph_Part}  
Given two integers $a$ and $b$, let $d$ = $\gcd(a,b)$. Let $C_b(a)$ denote a circulant graph where its vertex set $V = \{0,1, \ldots b-1\}$.  Then there exists a partition $\{S_1, \cdots, S_d\}$ of $V$ such that the following statements are true:
\begin{enumerate}
	\item The induced subgraphs $G[S_1], \cdots, G[S_d]$ are mutually disconnected components.
	\item Each $G[S_i]$, $i=1,\ldots,d$ induces an isometric cycle of equal length in $C_b(a)$.
\end{enumerate}
In other words, the vertex set $V$ as well as the edge set $E$ of $C_b(a)$ are partitioned by isometric cycles $C_1, \cdots, C_d$ of equal length.
\end{theorem}
Theorem \ref{thm:gcd} is used as a tool to prove Theorem \ref{thm:Cir_Graph_Part}. In the same way, Theorem \ref{thm:Cir_Graph_Part} will be used as a tool to prove  Theorem \ref{thm:VT_Acute_Edgeset_Part}.
\begin{figure}[ht!]
	\begin{center}
		\scalebox{0.6}{\includegraphics{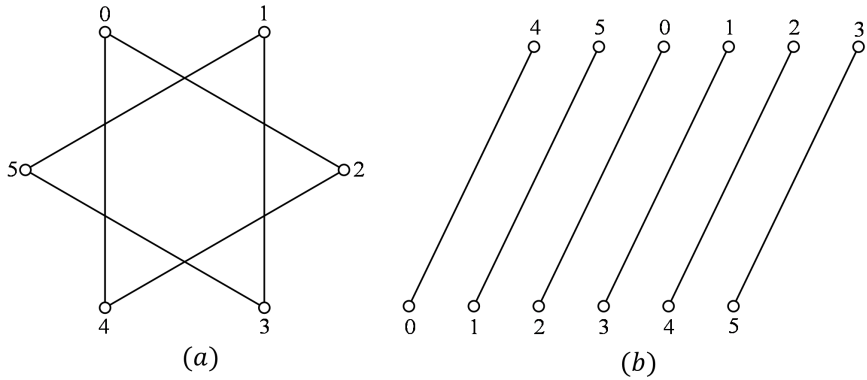}}
	\end{center}
	\caption{$(a)$ Circulant graph $C_6(4)$. $(b)$ The edges of $C_6(4)$ are arranged in a linear order.}
	\label{fig:BW_Circulant_Linear_Edges_4x6S1}
\end{figure}

\begin{theorem}
	\label{thm:VT_Acute_Edgeset_Part}
	Let $d = \gcd(r,s)$ of $\VT(r,s)$ where $s$ is larger than $r$. Then,
	\begin{enumerate}
	\item 
	Exactly $d$ members of $\{\AC(0,2k) : k = 0,1 \ldots s-1\}$ are distinct. 
	\item
	The vertex set $V$ of $\VT(r,s)$ is partitioned by  those $d$  distinct members of $\{\AC(0,2k) : k = 0,1 \ldots s-1\}$.
	\item
	The acute edge set $E_a$ of $\VT(r,s)$ is partitioned by  those $d$ distinct members of $\{\AC(0,2k) : k = 0,1 \ldots s-1\}$. 
	\end{enumerate}
Similar results are true for obtuse cycles $\OC(0,2k)$ and obtuse edge set $E_o$. 
\end{theorem}
\begin{proof}
	Since each path $\AP(0,2k)$ contains only acute edges, any two members of $\{\AP(0,2k) : k = 0,1 \ldots s-1\}$ are either mutually disjoint or isomorphic. By Property \ref{pro:VT_trivial_properties}, the vertex set $V$ and acute edge set $E_a$ of $\VT(r,s)$ are partitioned by paths \{$\AP(0,2k) : k = 0,1$, $\ldots$, $s-1$\}. 
	The graph $C_b(a)$ and its edges are displayed in Figure \ref{fig:BW_Circulant_Linear_Edges_4x6S1}. we will show that the acute paths $\{\AP(0,2k) : k = 0,\ldots,s-1\}$ are equivalent to the edges of $C_b(a)$. This is achieved by means of the following steps:	
\begin{description}
	\item[Step 1:] 
	First, all the obtuse edges are removed from $\VT(r,s)$. Refer to Figure \ref{fig:BW_Circulant_Linear_Edges_4x6S0}.
	\item[Step 2:] 
	Path $\AP(0,2k)$ is replaced by  edge  $(0,2k)(0, 2k + 2r)$ for $k = 0,\ldots,s-1$ where $(0,2k)$ and $(0, 2k + 2r)$ are the endpoints of $\AP(0,2k)$. This is achieved in two stages which are illustrated in Figure  \ref{fig:BW_Circulant_Linear_Edges_4x6S2}. 
	\item[Step 3:] 
	Label $(0,2k)$ are replaced by label $k$ for $k = 0,\ldots,s-1$. Refer to Figure \ref{fig:BW_Circulant_Linear_Edges_4x6S3}. 
	\item[Step 4:] The resultant  graph is isomorphic to $C_6(4)$ which is displayed in Figure \ref{fig:BW_Circulant_Linear_Edges_4x6S1}.
\end{description}

 Now the theorem follows from Theorem \ref{thm:Cir_Graph_Part}. 
\end{proof}

\begin{figure}[ht!]
	\begin{center}
		\scalebox{0.74}{\includegraphics{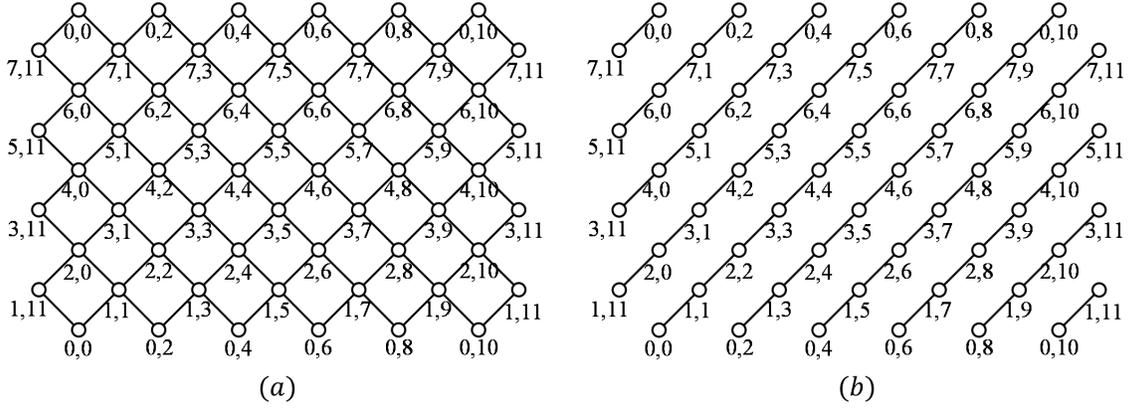}}
	\end{center}
	\caption{$(a)$ Villarceau torus $\VT(4,6)$. $(b)$ The obtuse edges are removed from $\VT(4,6)$.}
	\label{fig:BW_Circulant_Linear_Edges_4x6S0}
\end{figure}

\begin{figure}[ht!]
	\begin{center}
		\scalebox{0.6}{\includegraphics{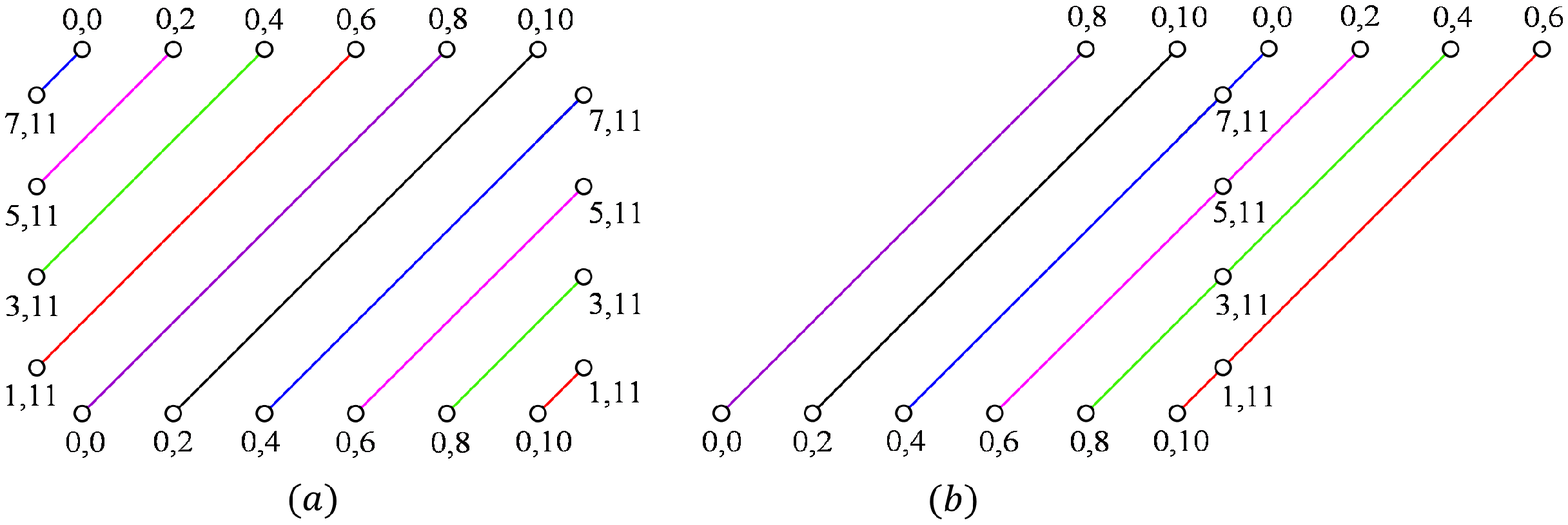}}
	\end{center}
	\caption{$(a)$ After removing obtuse edges from $\VT(r,s)$, In $\VT(r,s)$, the end vertices $(0,2k)$ and $(0, 2k + 2r)$ of $\AP(0,2k)$ are replace an edge $(0,2k)(0, 2k + 2r)$. Some lines  between $(0,2k)$ and $(0, 2k + 2r)$ are broken into two segments which are distinguished by different colors. $(b)$ The broken segments are merged into one straight line.}
	\label{fig:BW_Circulant_Linear_Edges_4x6S2}
\end{figure}
\begin{figure}[ht!]
	\begin{center}
		\scalebox{0.6}{\includegraphics{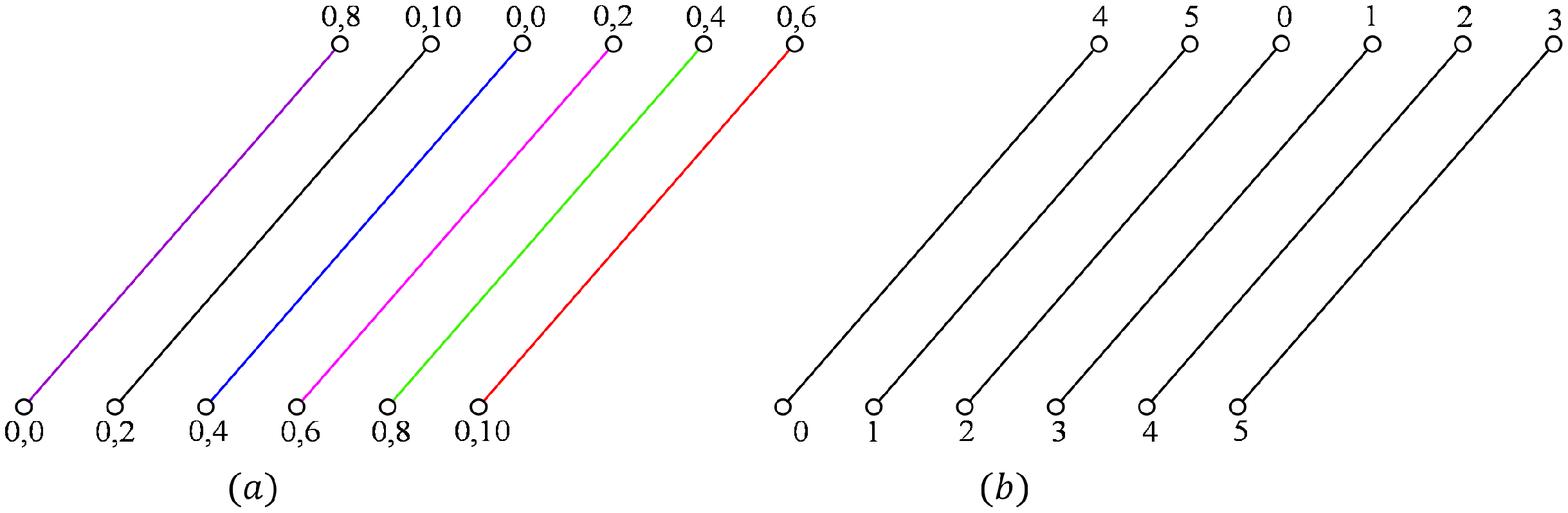}}
	\end{center}
	\caption{$(a)$ Now the acute path $\AP(0,2k)$ is replaced by an edge. $(b)$ The label $(0,2k)$ is replaced by $k$ for $k = 0,\ldots,s-1$. This is isomorphic to $C_6(4)$ which is described in \ref{fig:BW_Circulant_Linear_Edges_4x6S1}$(b)$.}
	\label{fig:BW_Circulant_Linear_Edges_4x6S3}
\end{figure}

Now we define a few more terms which we shall use in this section. First we define row and columns of $\VT(r,s)$:

 $ROW_{2i} = \{(2i,0), (2i,2) \ldots (2i,2k) \ldots (2i,2(s-1)) : i = 0, 1 \ldots r-1\}$
 
 $ROW_{(2i+1)} = \{(2i+1,1), (2i+1,3) \ldots (2i+1,2k+1) \ldots (2i+1,2s-1) : i = 0, 1  \ldots r-1\}$
 
 $COL_{2j} = \{(0, 2j), (2,2j) \ldots (2k,2j) \ldots (2(r-1), 2j) : j = 0, 1 \ldots s-1\}$
 
 $COL_{(2j+1)} = \{(1,2j+1), (3, 2j+1) \ldots (2k+1, 2j+1) \ldots (2r-1, 2j+1) : j = 0, 1 \ldots s-1\}$ 
 
 Given $ROW_i$ and $COL_j$, a cycle $C$ in $\VT(r,s)$ is said to have $p$ poloidal revolutions and $q$ toroidal revolutions if $C$ passes through $ROW_i$ by $p$ times and passes through $COL_j$ by $q$ times. Mathematically, given $ROW_i$ and $COL_j$, a cycle $C$ is said to have $p$ poloidal revolutions and $q$ toroidal revolutions if $|C \cap ROW_i| = p$ and $|C \cap COL_j| = q$.
\begin{theorem}
	\label{thm:Helix_p_q_revolutions}
	Given $\VT(r,s)$, let $d = \gcd(r,s)$, $p=r/d$ and $q=s/d$ where $s$ is assumed to be larger than $r$. Then, cycle $\AC(0,2k)$ has $p$ toroidal revolutions and $q$ poloidal revolutions. 
\end{theorem} 
\begin{proof}
	Consider two columns $COL_{j_1}$ and $COL_{j_2}$ such that $j_1 \neq  j_2$. Since $\AC(0,2k)$ contains only acute edges, $|\AC(0,2k) \cap COL_{j_1}| = |\AC(0,2k) \cap COL_{j_2}|$. Here, $\AC(0,2k) \cap COL_{j_1}$ is the set of vertices common to both $\AC(0,2k)$ and $COL_{j_1}$. By Theorem \ref{thm:VT_Acute_Edgeset_Part}, we know that $|\AC(0,2k_1)| = |\AC(0,2k_2)|$ for $k_1 \neq k_2$. Also, by Theorem \ref{thm:VT_Acute_Edgeset_Part}, the number of distinct cycles in $\{\AC(0,2k) : k = 0,1 \ldots s-1\}$ is $d$. Since $COL_j$ has $r$ vertices and there are $d$ distinct cycles in $\{\AC(0,2k) : k = 0,1 \ldots s-1\}$, each cycle $\AC(0,2k)$ shares exactly $r/d$ vertices with $COL_j$. It means that $\AC(0,2k)$ has $p$ toroidal revolutions. Similarly, $\AC(0,2k)$ has $q$ poloidal revolutions.
\end{proof}
\begin{figure}[ht!]
	\begin{center}
		\scalebox{1.0}{\includegraphics{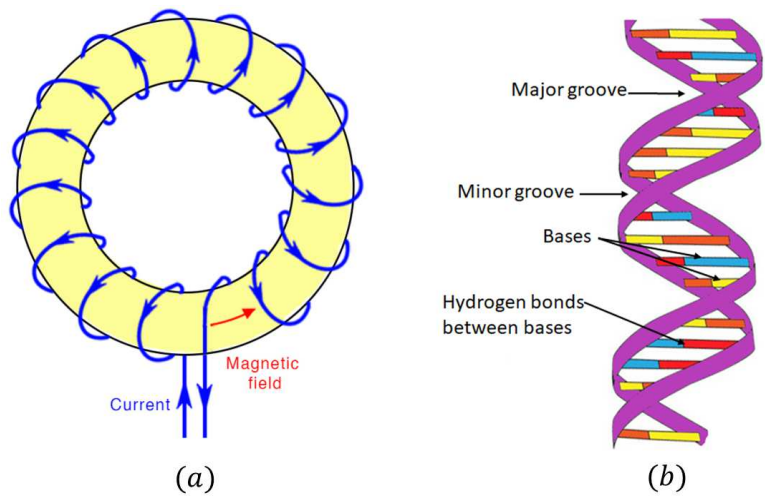}}
	\end{center}
	\caption{$(a)$ Toroidal Magnetic Field. The diagram is retrieved from https://www.sciencefacts.net/toroid-magnetic-field.html $(b)$ Double stranded helix model of DNA. This diagram is retrieved from https://www.vedantu.com/.}
	\label{fig:Pysics_Biology_Application}
\end{figure}

The Heat Transfer in Laminar Flows is modeled by means of toroidal helix \cite{ZhNa20}. The models of the electron illustrate that electrons internally move around a torus \cite{Hest20}. The electrons move on helix trajectories in a uniform magnetic field. It is called Helical Toroidal Electron Model \cite{Cons18, ZhNa20}. Refer to Figure  \ref{fig:Pysics_Biology_Application} $(a)$. \ref{thm:Helix_p_q_revolutions} provides the mathematical interpretation of the Helical Toroidal Electron Model. 

Molecular biologists have modeled DNA strands in terms toroidal helices and call it as double stranded helix model of DNA. A DNA molecule consists of two strands that twist around one another to form a helix. Refer to Figure  \ref{fig:Pysics_Biology_Application} $(b)$.  This twisted-ladder structure double helix (DNA) was discovered by James Watson and Francis Crick in 1953 \cite{WaCr53}.  The double helix structure of DNA contains a major groove and minor groove which are similar to acute and obtuse cycles in $\VT(r,s)$ \cite{Huret06}. 

The toroidal helix in Physics and Molecular Biology is continuous whereas the toroidal helix in Villarceau torus is discrete. 
\section{Isometric and convex cycles in $\VT(r,s)$}
\label{sec:is-convex-cycles-VT}
\begin{theorem}
	\label{thm:iso-convex-cycles-VT(r,r)}
	Cycles $\AC(0,2k)$ and $\OC(0,2k)$, $k = 0,\ldots,r-1$, are isometric but not convex in $\VT(r,r)$. 
\end{theorem}
\begin{proof}
	In the square Villarceau torus $\VT(r,r)$, the end vertices  of path $\AP(0,2k)$ are $(0,2k)$ and $(0,2k+2r) = (0,2k)$. Since the end vertices of  path $\AP(0,2k)$ are the same, it becomes a cycle and isometric. Notice that this is not true for $\VT(r,s)$ where $r \neq s$.
	
	Consider cycles $\AC(0,0)$ and $\OC(0,0)$. They do not have any edges in common. Moreover, $\AC(0,0)$ and $\OC(0,0)$ are of same size and meet only at two vertices $(0,0)$ and $(r,r)$ which are diagonally opposite. In other words, the cycles $\AC(0,0)$ and $\OC(0,0)$ are orthogonal as in Figure \ref{fig:NO_Convex_Cycles_5x5}. Thus, cycles $\AC(0,0)$ and $\OC(0,0)$ are not convex. Since $\VT(r,r)$ is symmetric, we can extend the arguments to show that $\AC(0,2k)$ and $\OC(0,2k)$, $k = 0,1 \ldots s-1$, are not convex. 
\end{proof}
\begin{figure}[ht!]
	\begin{center}
		\scalebox{0.9}{\includegraphics{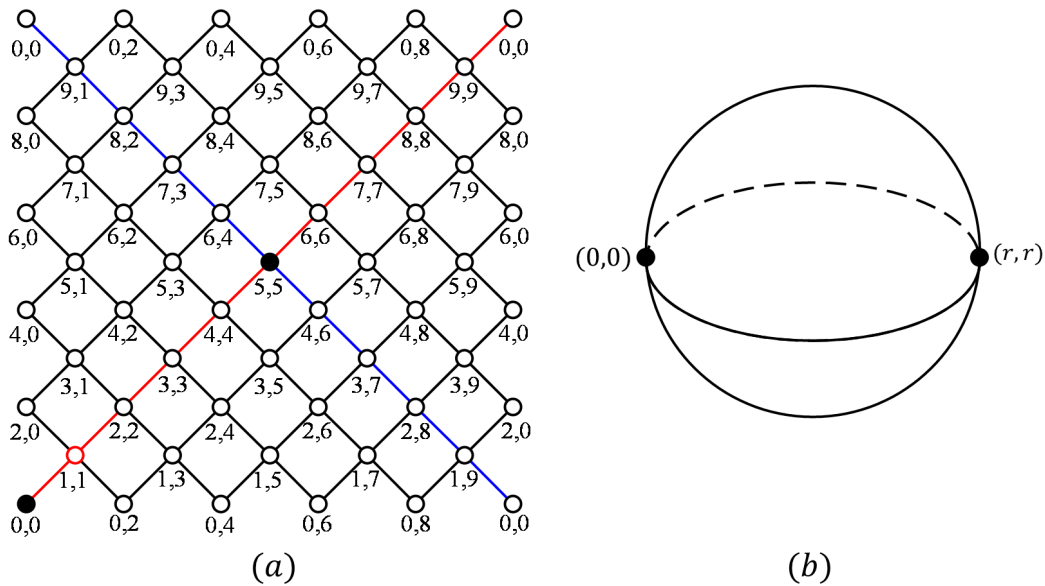}}
	\end{center}
	\caption{$(a)$ Cycles $\AC(0,0)$ is marked in red and $\OC(0,0)$ is marked in blue. $(b) $The two cycles intersect only at vertices $(0,0)$ and $(r,r)$ which are diametrically opposite. Since $|\AC(0,0)| = |\OC(0,0)|$,  $\AC(0,0)$ and $\OC(0,0)$ are not convex cycles.}
	\label{fig:NO_Convex_Cycles_5x5}
\end{figure}
We leave a remark that Theorem \ref{thm:iso-convex-cycles-VT(r,r)} is true for $\VT(r,s)$ where $s$ is a multiple of $r$. 
The next result is interesting that Cycles $\AC(0,2k)$ and $\OC(0,2k)$ are not isometric in $\VT(r,s)$, when $s$ is not a multiple of $r$.  
\begin{figure}[ht!]
	\begin{center}
		\scalebox{0.8}{\includegraphics{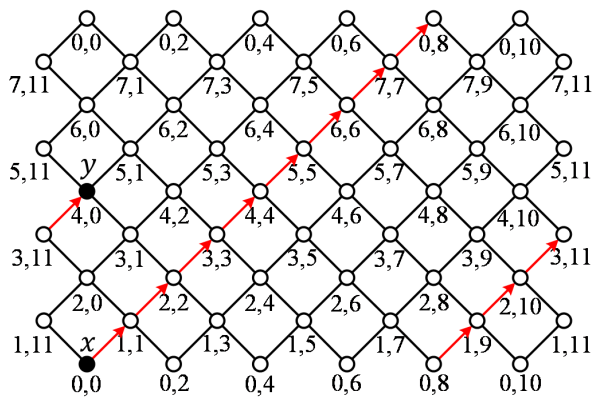}}
	\end{center}
	\caption{Here $x = (2r_1,0)$ and $y = (2r_2,0)$. The distance between $x$ and $y$ along $COL_0$ is strictly less than the distance between $x$ and $y$ along $\AC(0,2k)$.}
	\label{fig:No_Isometric_cycle4x6}
\end{figure}
\begin{theorem}
	\label{thm:AC_not_isometric_VT(r,s)}
	Cycles $\AC(0,2k)$ and $\OC(0,2k)$, $k = 0,\ldots,r-1$, are not isometric in $\VT(r,s)$, when $s$ is not a multiple of $r$.  
\end{theorem}
\begin{proof}
Let $d = \gcd(r,s)$. By Theorem \ref{thm:Helix_p_q_revolutions}, $\AC(0,2k)$ has $p = r/d$ toroidal revolutions. Since $s$ is not a multiple of $r$, $d \lneq r$ and hence $p \geq 2$. Consider column $COL_0$ in $\VT(r,s)$. Then $\AC(0,2k)$ intersects $COL_0$ with at least two vertices, say $(2r_1,0)$ and $(2r_2,0)$. Let $x = (2r_1,0)$ and $y = (2r_2,0)$. Now $d_ {COL_0}(x, y)$ $\leq r$ because $|COL_0| = 2r$. Also, $d_{\AC(0,2k)}(x, y) \geq 2s > s > r$. Refer to Figure \ref{fig:No_Isometric_cycle4x6}.
Here $d_ {COL_0}(x, y)$ and $d_{\AC(0,2k)}(x, y)$ are the distances between $x$ and $y$ along $COL_0$ and $\AC(0,2k)$ respectively. Thus, $\AC(0,2k)$  and similarly $\OC(0,2k)$, $k = 0,\ldots,r-1$, are not isometric. 
\end{proof}
Now we shall demonstrate another striking property of Villarceau torus that $\VT(r,s)$ has no convex cycles other than $4$-cycles $C_4$. This is true for any $r$ and $s$.
 
\begin{theorem}
	\label{No_Convex_Cycles}
	Villarceau torus $\VT(r,s)$ has no convex cycles other than $C_4$. 
\end{theorem} 
\begin{proof}
	Cycles in $\{\AC(0,2k) : k = 0,1 \ldots s-1\}$ are the only cycles which have only acute edges. In the same way, cycles in $\{\OC(0,2k): k = 0,1 \ldots s-1\}$ are the only cycles which have only obtuse edges. By Theorem \ref{thm:iso-convex-cycles-VT(r,r)} and  \ref{thm:AC_not_isometric_VT(r,s)}, $\AC(0,2k)$ and $\OC(0,2k)$ are not convex. 
	
	Suppose there is a cycle $C_k$, $k \neq 4$, in $\VT(r,s)$ such that $C_k$ has both acute and obtuse edges. It means that there exist two adjacent edges $ab$ and $bc$ in $C_k$ such that $ab$ is acute and $bc$ is obtuse. These two edges $ab$ and $bc$ lie in some $C_4 = \{a,b,c,d\}$. Since $C_4$ does not lie inside any cycle, $d$ does not belong to $C_k$. Now vertices $a$ and $c$ have an isometric path $adc$ which do not lie inside $C_k$. Thus, $C_k$ is not convex. 
\end{proof}
\begin{remark}
	The poloidal and toroidal cycles of ring torus $C_r\cp C_s$ are convex whereas $\VT(r,s)$ has no convex cycles other than $C_4$.
\end{remark}

\section{Villarceau torus does not admit convex edgecuts}
\label{sec:VT-no-convex-edgecut}
A set $S$ of edges is said to be a \textit{$U$-$W$ edgecut} if $G\setminus S$ has exactly two connected components $G[U]$ and $G[W]$ where $V =  U \cup W$, $U \neq \emptyset$ and $W \neq \emptyset$. 
When both the disconnected components $G[U]$ and $G[W]$ are convex, the edgecut $S$ is said to be a \textit{convex edgecut} of $G$.
For the sake of clarity, whenever it is required, we call an edgecut as a $(U,W)$-edgecut.

\begin{theorem}
	\label{thm:VT-no-convex-edgecut}
	Villarceau torus $\VT(r,s)$ has no convex edgecuts. 
\end{theorem}
\begin{proof}
	Suppose there exists a $U$-$W$ edgecut $S$ of $\VT(r,s)$ such that $U \neq \emptyset$ and $W \neq \emptyset$. Then, there exists an edge $ab$ from the edgecut $S$ such that $a \in U$ and $b \in  W$. 
	Now there are two cases that edge $ab$ is either acute or obtuse. Let us consider the case that $ab$ is an obtuse edge. For the sake of simplicity, we assume that vertex $a$ is $(r,r)$. This is possible. Since $\VT(r,s)$ is symmetrical, after fixing that $a =(r,r)$, the origin and the coordinate system of $\VT(r,s)$ may be decided. Since $b$ is adjacent to $a$, there are two cases that $b$ is either $(r-1,r+1)$ or $(r+1,r-1)$. Let us consider the case that $b$ is $(r-1,r+1)$.  It means that $b \in \AC(0,2)$.
	\begin{figure}[ht!]
		\begin{center}
			\scalebox{0.9}{\includegraphics{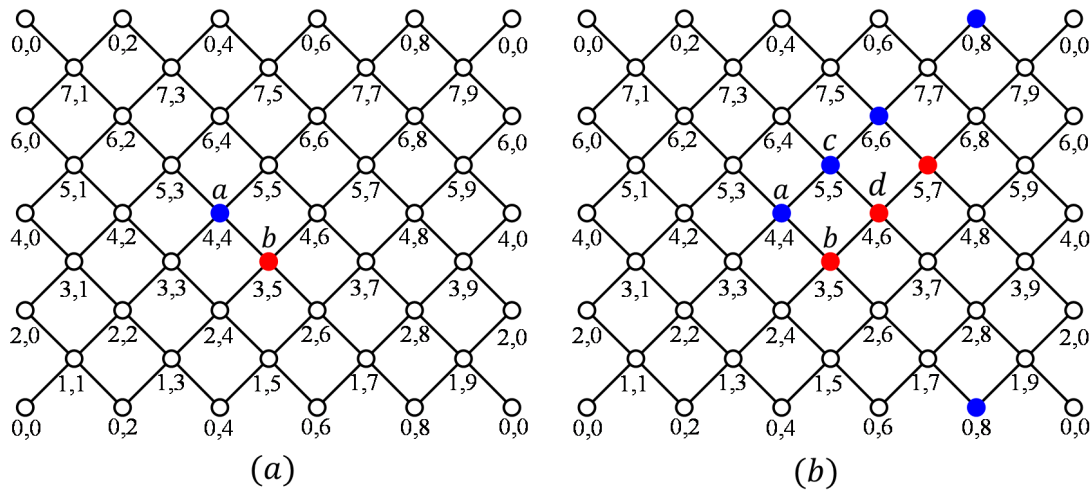}}
		\end{center}
		\caption{The vertices of $U$ are in blue and the vertices of $W$ are in red. $(a)$ $a \in U \cap \AC(0,0)$ and $b \in W \cap \AC(0,2)$. $(b)$ $\AC(0,0) \subseteq U$ and $\AC(0,2) \subseteq W$. Notice that a red vertex lies on an isometric path between two blue vertices.} 
		\label{fig:No_Convex_Edgecut}
	\end{figure}
	Now, $a \in U \cap \AC(0,0)$ and $b \in W \cap \AC(0,2)$. Let $c = (r,r+2)$ and $d = (r+1,r+1)$. Observe that $d$ is adjacent to $a$ and lies in $\AC(0,0)$. In the same way, $c$ is adjacent to $b$ and lies in $\AC(0,2)$. Consider $C_4$ cycle $\{a,b,c,d\}$. Since $a$ is in $U$ and $b$ is in $W$, vertex $d$ must be in $U$ and vertex $c$ must be in $W$. Otherwise, it will contradict our assumption that $U$ and $W$ are convex. In other words, when $a \in U \cap \AC(0,2k)$ and $b \in W \cap \AC(0,2k+2)$, its adjacent  vertices $d \in U \cap \AC(0,2k)$ and $c \in W \cap \AC(0,2k+2)$. Continuing the same logic, we show that $\AC(0,0)$ is a subset of $U$ and $\AC(0, 2)$ is a subset of $W$. 
	
	Now consider the vertex $(r+r,j+r)$ which is $(2r,j+r)$ =  $(0,j+r)$. Notice that $(0,j+r)$ is in $U \cap AC(0,2k)$. Now consider the obtuse path $\OP((r,r),(0,j+r))$ between $(r,r)$ and $(0,j+r)$. Since the length of $\OP((r,r),(0,j+r))$ is $r$, it is isometric. Now vertex $b = (r-1,r+1) \in W$ and lies in the isometric path $\OP((r,r),(0,j+r))$ where $(r,r)$ and $(0,j+r)$ are in $U$. It implies that $U$ is not convex. This is a contradiction. 
	
	Initially we assumed that $ab$ is an obtuse edge. If $ab$ is an acute edge, we will argue with obtuse cycles $\OC(0,0)$ and $\OC(0, 2)$ in the place of $\AC(0,0)$ and $\AC(0,0)$. The other case $b$ is $(r+1,r-1)$ is also straightforward.  
\end{proof}
\section{Extension of the cut method}
\label{sec:convex-cut}
{\red  UNDER PREPARATION}

The combinatorial graph problems such as embedding of one graph into another graph, computation of Wiener Index, estimating forwarding indices or designing routing algorithms require convex edgecuts. The cut method which is built on convex edgecuts is a powerful tool to solve some graph combinatorial problems such as average distance of a graph, edge forwarding index, embedding and routing algorithms etc \cite{Klavzar08, KlNa14, MaRa08}. The basic requirement of the cut method is a convex edgecut partition in some forms \cite{KlRo14}. The cut method was introduced by Klavžar el al \cite{KlGu95}. It was further extended  by Chepoi el al \cite{ChDe97} in the form of $\ell_1$-graphs. The $k$-cut method was designed in the form of quotient graphs by means of  $\Theta^*$-relation \cite{GrWi85, Klavzar08, KlNa14}. 

There are graphs which do not admit convex edgecuts. Notice that an odd cycle admits neither a convex edgecut nor an edgecut partition. However, there are graphs which admit edgecut partitions but do not admit a convex edgecut partition. Villarceau torus is one such graph. 
Here we extend the convex cut method to graphs which admit edgecut partitions but do not admit a convex edgecut partition.

A partition $\{E_1,\cdots, E_p\}$ of $E$ is said to be an \textit{edgecut partition} if each $E_i$ is an edgecut of $G$ for $i = 1, \ldots, p$. 
A partition $\{E_1,\cdots, E_p\}$ is said to be a \textit{convex edgecut partition} if each $E_i$ is a convex edgecut of $G$ for $i = 1, \ldots, p$.

In this paper, $\mathcal{C}$ denotes a set of paths in $G$. When $\mathcal{C}$ contains  $|V|(|V|-1)/2$ paths, it is called routing \cite{XuXu13}. The congestion $\Pi(G,\mathcal{C},e)$ of an edge $e$ is the number of paths of $\mathcal{C}$ passing through $e$. 
Given a set $S$ of edges, $\Pi(G,\mathcal{C},S)$ = $\sum_{e \in S}{\Pi(G,\mathcal{C},e)}$.
Also, we define
$\Pi(G,\mathcal{C}) = \max_{e \in E} \Pi(G,\mathcal{C},e)$ 
and 
$\Pi(G) = \min_{\mathcal{C} \in G} \Pi(G,\mathcal{C})$

\begin{theorem}
	\label{thm:Sum-Paths}
Let $\mathcal{C}$ denote a set of paths in a graph $G$.  Let $\{E_1,\cdots,E_p\}$ be an edgecut partition of $G$ where each $E_i$ is a $(U_i,W_i)$-edgecut of $G$. Given $\mathcal{C}$ and $\{E_1,\cdots,E_p\}$ of $G$, define the following terms:

$k_{U_i,W_i}$ = $\sum_{x \in U_i, y \in W_i, P(x,y)\in \mathcal{C}}{|E_i\cap P(x,y)|}$
\\

$k_{U_i}$ = $\sum_{x,y \in U_i, P(x,y)\in \mathcal{C}}{|E_i\cap P(x,y)|}$
\\

$k_{W_i}$ = $\sum_{x,y \in W_i, P(x,y)\in \mathcal{C}}{|E_i\cap P(x,y)|}$
\\

\noindent
where $|E_i\cap P(x,y)|$ is the number of edges common to $E_i$ and $P(x,y)$. Then
$$\sum_{P(x,y)\in \mathcal{C}}{|P(x,y)|} = \sum_{i=1}^{p}{(k_{U_i,W_i}+k_{U_i}+k_{W_i})}$$
\end{theorem}
\begin{proof}
	It is well-known \cite{XuXu13} that $\sum_{P(x,y)\in \mathcal{C}}{|P(x,y)|}$ = $\sum_{i=1}^{p}{\Pi(G,\mathcal{C},E_i)}$.
	Consider $E_i$ where $E_i$  is a $(U_i,W_i)$-edgecut. Initially $\Pi(G,\mathcal{C},E_i) = 0$. Now, there are only three possibilities:
	\begin{enumerate}
		\item  For $x \in U_i$ and $y \in W_i$, $\Pi(G,\mathcal{C},E_i)$ is incremented by $k_{U_i,W_i}$.
		\item  For $x,y \in U_i$, $\Pi(G,\mathcal{C},E_i)$ is incremented by $k_{U_i}$.
		\item  For $x,y \in W_i$, $\Pi(G,\mathcal{C},E_i)$ is incremented by $k_{W_i}$.
	\end{enumerate}
	Thus, $\Pi(G,\mathcal{C},E_i)$ = $k_{U_i,W_i}+k_{U_i}+k_{W_i}$. Hence, $\sum_{P(x,y)\in \mathcal{C}}{|P(x,y)|}$ = $\sum_{i=1}^{p}{\Pi(G,\mathcal{C},E_i)}$ = $\sum_{i=1}^{p}{(k_{U_i,W_i}+k_{U_i}+k_{W_i})}$
\end{proof}

\begin{figure}[ht!]
	\begin{center}
		\scalebox{0.8}{\includegraphics{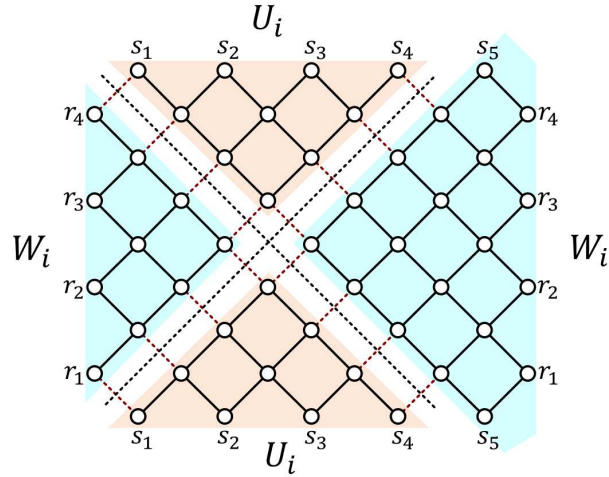}}
	\end{center}
	\caption{$X$ type of edgecuts. The edges of edgecut $S$ are depicted in dotted lines. The deletion of S disconnects $G$ into two connected components $G{U}$ and $G[W]$. The disconnected components $G{U}$ and $G[W]$ are in two different colors.  }
	\label{fig:WI_UW_Partition_4x5}
\end{figure}
\begin{figure}[ht!]
	\begin{center}
		\scalebox{0.85}{\includegraphics{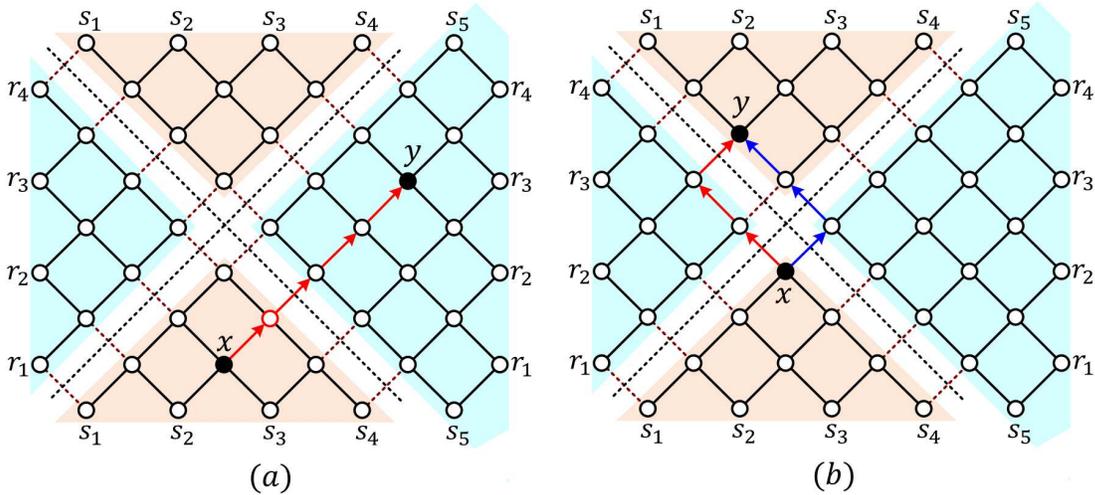}}
	\end{center}
	\caption{to be added.}
	\label{fig:WI_UW_Partition_Case1_2_4x5}
\end{figure}
\begin{figure}[ht!]
	\begin{center}
		\scalebox{0.85}{\includegraphics{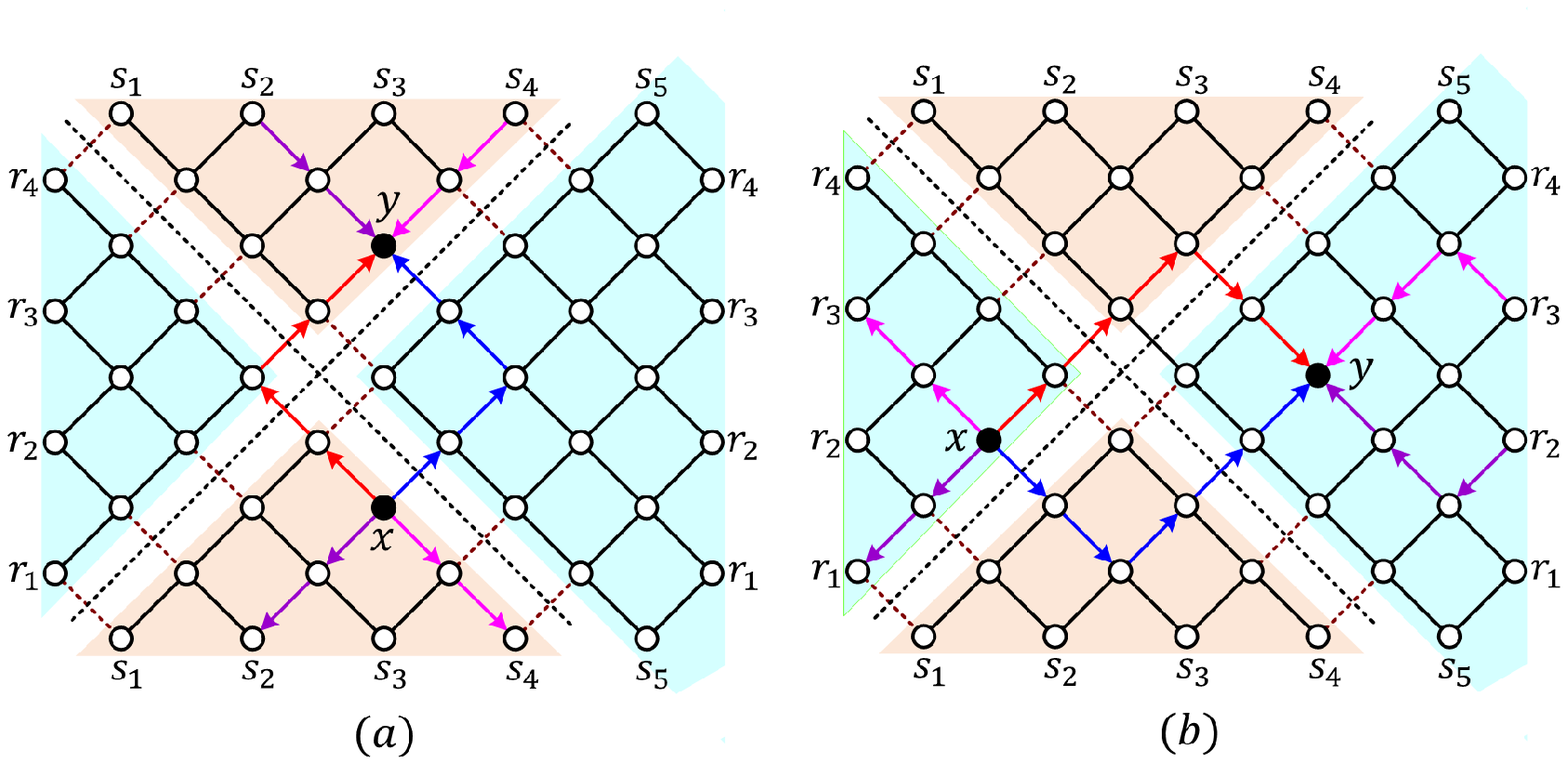}}
	\end{center}
	\caption{.}
	\label{fig:WI_UW_Partition_Case3_4x5}
\end{figure}
\section{Computation of Wiener Index of $\VT(r,s)$}
\label{sec:WI}
{\red  UNDER PREPARATION}

Since Villarceau torus does not admit convex partitions, the computation of the Wiener Index of $\VT(r,s)$ becomes tedious and complicated. Many wonder “Is it worth and required, when an $O(n^3)$ brute-force algorithm \cite{Klavzar08,KlNa14,KlRo14,MaRa08} is available to compute Wiener Index of  $\VT(r,s)$?”. The implementation of this brute-force algorithm gives only the value of the Wiener Index and it does not give the formula for the Wiener Index of $\VT(r,s)$.
The objective here is to derive a formula for the Wiener Index of $\VT(r,s)$ so that this formula will be applied in other combinatorial problems such as graph embedding, forwarding indices and network routing algorithms \cite{MaRa08}. In addition, once a formula is derived, the running time of the computation of the Wiener Index becomes $O(1)$ which is constant time.

Even though $\VT(r,s)$ does not admit any convex edgecuts, it admits edgecuts. Since $\VT(r,s)$ is bipartite, there are several type of edgecuts. One  type of edgecuts is illustrated in Figure \ref{fig:WI_UW_Partition_4x5}.

Given a pair $x$ and $y$ of vertices in $G$, an \textit{acute isometric $x,y$-path} is an isometric path traversing only along acute edges and an \textit{obtuse isometric $x,y$-path} is an isometric path traversing only along obtuse edges. An \textit{acute-obtuse isometric $x,y$-path} is an isometric path that traverses  along acute edges first and then traverses along obtuse edges until reaching $y$. Similarly, an \textit{obtuse-acute isometric $x,y$-path} is an isometric path that traverses along  obtuse edges first and then traverses along acute edges until reaching $y$. 

We partition the set $V\times V$ into three subsets as follows:
$T_1 = \{{x,y} : x,y \in V$ and $\}$
xxxx
\begin{observation}
	\label{obs:WI-U_i-W_i}
Let $x$ and $y$ be a pair of vertices in $\VT(r,s)$. Let $\{E_1,\cdots,E_p\}$ be an arbitrary edgecut partition of $G$ where each $E_i$ is a $(U_i,W_i)$-edgecut of $G$. 
\begin{description}
	\item[Case 1: Suppose vertices $x$ and $y$ lie on  some $\AP(0,2k)$ (resp. $\OP(0,2k)$).]
	In this case, the isometric $x,y$-path $P(x,y)$ is an acute (resp. obtuse) isometric path. If $x,y \in U_i$ (resp. $W_i$), then $|E_i \cap P(x,y)|=0$. If $x \in U_i$ and $y \in W_i$, then $|E_i \cap P(x,y)|=1$. 
	
	\item[Case 2: Suppose $d(x,y) \neq  r$ and $d(x,y) \neq  s$.]
	In this case, there is one acute-obtuse isometric $x,y$-path  $P_1(x,y)$ and one obtuse-acute isometric $x,y$-path $P_2(x,y)$ between $x$ and $y$. These two isometric paths do not have common vertices other than $x$ and $y$. Also, $P_1(x,y)$ and $P_2(x,y)$ form a diamond shape between $x$ and $y$. If $x,y \in U_i$ (resp. $W_i$), then $|E_i \cap P_1(x,y)|= |E_i \cap P_2(x,y)| = 0$ or $|E_i \cap P_1(x,y)|= |E_i \cap P_2(x,y)|= 2$. If $x \in U_i$ and $y \in W_i$, then $|E_i \cap P_1(x,y)|= |E_i \cap P_2(x,y)|= 1$. 
	
	\item[Case 3: Suppose $d(x,y) = r$ \& $d(x,y) \neq s$ (resp. $d(x,y) = s$ \& $d(x,y) \neq r)$.]
	In this case, there is a pair of acute-obtuse  isometric $x,y$-path $P_1(x,y)$ and obtuse-acute  isometric $x,y$-path $P_2(x,y)$ which form a diamond shape between $x$ and $y$. Then, there is one more pair of acute-obtuse  isometric $x,y$-path $P_3(x,y)$ and obtuse-acute isometric $x,y$-path $P_4(x,y)$ which form another diamond shape between $x$ and $y$. All these four isometric paths are mutually vertex-disjoint except at $x$ and $y$. If $x,y \in U_i$ (resp. $x,y \in W_i$), then $|E_i \cap P_1(x,y)|=|E_i\cap P_2(x,y)| = 2$ and $|E_i \cap P_3(x,y)|=|E_i\cap P_4(x,y)| = 0$. 
	If $x \in U_i$ and $y \in W_i$, 
	then $|E_i \cap P_1(x,y)| = |E_i\cap P_2(x,y)| = |E_i \cap P_3(x,y)| = |E_i\cap P_4(x,y)| = 1$. 
	
\end{description}
\end{observation}

In Observation \ref{obs:WI-U_i-W_i}, the case “$d(x,y) = r  = s$” does not arise in the rectangular Villarceau torus $\VT(r,s)$ because $s$ is assumed to be larger than $r$.
\section{Congestion-balanced routing}
\label{sec:cong-balanced-routing}
{\red  UNDER PREPARATION}

The problem of finding optimal congestion-balanced routing in a graph G is equivalent to the problem of finding the edge forward index of $G$ \cite{XuXu13}. Thus, the problem of finding a congestion-balanced routing for general graphs is NP-complete \cite{XuXu13}.

A routing $\mathcal{C}$ is said to be \textit{congestion-balanced} if
$\Pi(G,\mathcal{C},e_i)$ = $\Pi(G,\mathcal{C},e_j)$
for every $e_i$ and $e_j$ in $E(G)$.
A routing $\mathcal{C}$ is said to be optimal congestion-balanced if 
$\Pi(G,\mathcal{C},e_i)$ = $\Pi(G,\mathcal{C},e_j)$ = 
$\dfrac{d_{u,v\in V}d_G(u,v)}{|E|}$ for every $e_i$ and $e_j$ in $E(G)$.

Now we shall demonstrate how Villarceau torus $\VT(r,s)$ admits an optimal congestion-balanced routing.

\section*{Acknowledgment}
I’d like to acknowledge Prof Indra Rajasingh, Dr Parthiban Natarajan, Mr Andrew Arokiaraj and Dr Prabhu Saveri for their contributions to this research.

\end{document}